\documentclass[11pt]{article}

\usepackage[margin=1.15in]{geometry}
\usepackage{amsmath,amssymb,amsthm,mathtools}
\usepackage{enumitem}
\usepackage[colorlinks=true,linkcolor=blue,citecolor=blue,urlcolor=blue]{hyperref}

\newtheorem{theorem}{Theorem}[section]
\newtheorem{lemma}[theorem]{Lemma}
\newtheorem{proposition}[theorem]{Proposition}

\newtheorem{conjecture}[theorem]{Conjecture}

\theoremstyle{remark}

\title{An Improved Lower Bound for the de Bruijn--Erd\H{o}s Consecutive Gap Problem}

\author{Samuel Korsky}

\date{\today}

\begin{document}

\maketitle

\begin{abstract}
\noindent
Let $(x_n)_{n\geq 1}$ be a sequence of distinct points on the unit circle. After the first $n$ points are inserted, the circle is divided into $n$ intervals. For a fixed integer $r\geq 1$, let $M_n^{(r)}$ and $m_n^{(r)}$ denote respectively the largest and smallest total lengths of $r$ consecutive intervals. A theorem of de Bruijn and Erd\H{o}s gives
\[
    \limsup_{n\to\infty}\frac{M_n^{(r)}}{m_n^{(r)}}\geq 1+\frac1r .
\]
The case $r=1$ is sharp and gives the classical factor $2$. The cases $r\geq 2$ remain much less understood. We prove the improved lower bound
\[
    \limsup_{n\to\infty}\frac{M_n^{(r)}}{m_n^{(r)}}
    \geq
    1+\frac{r}{r^2-1}
    \qquad (r\geq 2).
\]
In particular, for two consecutive intervals the lower bound becomes $5/3$, improving the de Bruijn--Erd\H{o}s bound $3/2$.
\end{abstract}

\section{Introduction}

Let
\[
    x_1,x_2,x_3,\ldots
\]
be a sequence of distinct points on the circle $\mathbb T=\mathbb R/\mathbb Z$. After the first $n$ points are inserted, they divide the circle into $n$ intervals. We write these interval lengths in cyclic order as
\[
    g_1^{(n)},g_2^{(n)},\ldots,g_n^{(n)}.
\]
For a positive integer $r$, define
\[
    M_n^{(r)}
    =
    \max_i\left(g_i^{(n)}+g_{i+1}^{(n)}+\cdots+g_{i+r-1}^{(n)}\right)
\]
and
\[
    m_n^{(r)}
    =
    \min_i\left(g_i^{(n)}+g_{i+1}^{(n)}+\cdots+g_{i+r-1}^{(n)}\right),
\]
where indices are taken cyclically. Thus $M_n^{(r)}$ is the largest length of $r$ consecutive intervals and $m_n^{(r)}$ is the smallest length of $r$ consecutive intervals.

\bigskip
\noindent
The problem goes back to de Bruijn and Erd\H{o}s \cite{debruijn-erdos}. They proved that, for every sequence and every $r\geq 1$,
\[
    \limsup_{n\to\infty}
    \frac{M_n^{(r)}}{m_n^{(r)}}
    \geq
    1+\frac1r .
\]
For $r=1$, this says that the ratio between the largest and smallest interval length must be arbitrarily close to $2$ infinitely often; this is sharp. This case has also recently been revisited from a finite-horizon perspective by DeLeo, Henderschedt, and Wells \cite{deleo-henderschedt-wells}. 

\bigskip
\noindent
For $r\geq 2$, the sharp answer appears to be open. Recent work of Cl\'ement and Steinerberger \cite{clement-steinerberger} studies this problem under the name ``balanced stick breaking" and constructs sequences for which the corresponding ratio is at most
\[
    1+O\left(\frac{\log r}{r}\right).
\]
See also the expository slides of Gniecki \cite{gniecki-slides} for an accessible account of the de Bruijn--Erd\H{o}s problem and related quantities.

\bigskip
\noindent
The purpose of this note is to give a short improved lower bound. We prove the following.

\begin{theorem}\label{thm:main}
For every integer $r\geq 2$ and every sequence of distinct points on $\mathbb T$,
\[
    \limsup_{n\to\infty}
    \frac{M_n^{(r)}}{m_n^{(r)}}
    \geq
    1+\frac{r}{r^2-1}.
\]
\end{theorem}

\smallskip
\noindent
Since
\[
    \frac{r}{r^2-1}
    =
    \frac1r+\frac1{r(r^2-1)},
\]
this strictly improves the de Bruijn--Erd\H{o}s lower bound for every $r\geq 2$. For $r=2$, Theorem \ref{thm:main} gives
\[
    \limsup_{n\to\infty}\frac{M_n^{(2)}}{m_n^{(2)}}
    \geq
    \frac53 .
\]

\smallskip
\noindent
The proof exploits the local nature of the splitting dynamics: inserting a new point only splits one existing interval, so only the consecutive sums near that interval can change. We show that if the ratio $M_n^{(r)}/m_n^{(r)}$ is assumed to stay below a fixed $\rho$, then a split which does not significantly decrease $M_n^{(r)}$ creates a protected block of nearby intervals. None of the intervals in this block may be split again until $M_n^{(r)}$ has decreased by a definite multiplicative factor. Counting these protected blocks over one multiplicative epoch gives the desired obstruction.

\section{Notation and Basic Observations}

Fix $r\geq 2$. Since $r$ remains fixed throughout the proof, we suppress the superscript and write
\[
    M_n=M_n^{(r)},\qquad m_n=m_n^{(r)},\qquad R_n=\frac{M_n}{m_n}.
\]
An $r$-block means a union of $r$ consecutive intervals in the cyclic order.

\bigskip
\noindent
We shall repeatedly use the following elementary observation:

\begin{lemma}\label{lem:M-monotone}
The sequence $M_n$ is nonincreasing.
\end{lemma}

\begin{proof}
Suppose a gap $\ell$ is split into two gaps $x$ and $y$, where $x+y=\ell$. Consider any new $r$-block after the split.

\bigskip
\noindent
If the new block contains $x$ but not $y$, then replacing $x$ by $\ell$ gives an old $r$-block whose length is at least as large. The same applies if the new block contains $y$ but not $x$. If the new block contains both $x$ and $y$, then merging them gives an old block with $r-1$ gaps; extending it by one adjacent old gap gives an old $r$-block whose length is at least as large. Hence every new $r$-block has length at most $M_n$, and therefore $M_{n+1}\leq M_n$.
\end{proof}

\smallskip
\noindent
The average length of an $r$-block is exactly $r/n$, because each gap is counted in exactly $r$ such blocks. Hence
\[
    m_n\leq \frac rn\leq M_n.
\]
In particular, if $R_n\leq \rho$, then
\[
    M_n\leq \rho\cdot\frac rn .        \tag{2.1}\label{eq:M-upper-under-ratio}
\]

\section{A Protected Block Lemma}

The next lemma is the local input. It says that a split which does not reduce $M_n$ much creates a block of intervals that cannot be touched again until $M_n$ has fallen by a definite factor.

\bigskip
\noindent
Throughout this section assume that, from some point onward,
\[
    R_n\leq \rho .
\]
The proof of Theorem \ref{thm:main} will choose $\rho$ below the claimed lower bound and derive a contradiction.

\begin{lemma}[Protected block]\label{lem:protected-block}
Fix $\eta\in(0,1)$ and suppose that the step $n\to n+1$ satisfies
\[
    M_{n+1}\geq (1-\eta)M_n
\]
and
\[
    R_{n+1}\leq \rho.
\]
Suppose the split gap is $\ell=x+y$. After the split, consider the $2r$ consecutive gaps
\[
    h_1,h_2,\ldots,h_{2r},
\]
where
\[
    h_r=x,\qquad h_{r+1}=y,
\]
with $r-1$ gaps to the left of $\ell$ and $r-1$ gaps to the right of $\ell$ included.

\bigskip
\noindent
Set
\[
    q=\frac{1-\eta}{\rho},
    \qquad
    \alpha=1-q=\frac{\rho-1+\eta}{\rho}.
\]
Then
\[
    h_j\leq \alpha M_n
    \qquad (1\leq j\leq 2r).      \tag{3.1}\label{eq:gap-bound}
\]

\smallskip
\noindent
Moreover, suppose that at a later time $T$ one of the gaps $h_j$ is split, and that no gap among
\[
    h_1,\ldots,h_{2r}
\]
has been split before time $T$. If $R_T\leq \rho$, then
\[
    M_T\leq \beta M_n,
    \qquad
    \beta=(r-1)(\rho-1+\eta).       \tag{3.2}\label{eq:release-factor}
\]
\end{lemma}

\smallskip
\begin{proof}
First we prove \eqref{eq:gap-bound}. Since $M_{n+1}\geq (1-\eta)M_n$ and $R_{n+1}\leq\rho$, we have
\[
    m_{n+1}
    \geq
    \frac{M_{n+1}}{\rho}
    \geq
    \frac{1-\eta}{\rho}\cdot M_n
    =
    qM_n .
\]
Therefore every $r$-block after the split has length at least $qM_n$.

\bigskip
\noindent
For $1\leq i\leq r+1$, define the $r$-blocks
\[
    W_i=h_i+h_{i+1}+\cdots+h_{i+r-1}.
\]
Then
\[
    W_i\geq qM_n
    \qquad (1\leq i\leq r+1).       \tag{3.3}\label{eq:W-lower}
\]
For $1\leq i\leq r$, define
\[
    U_i=h_i+h_{i+1}+\cdots+h_{i+r}.
\]
Each $U_i$ contains both $x$ and $y$. If we merge $x+y$ back into the old gap $\ell$, then $U_i$ becomes an old $r$-block. Hence
\[
    U_i\leq M_n
    \qquad (1\leq i\leq r).         \tag{3.4}\label{eq:U-upper}
\]
For $1\leq j\leq r$,
\[
    h_j=U_j-W_{j+1}\leq M_n-qM_n=\alpha M_n.
\]
For $r+1\leq j\leq 2r$,
\[
    h_j=U_{j-r}-W_{j-r}\leq M_n-qM_n=\alpha M_n.
\]
This proves \eqref{eq:gap-bound}.

\bigskip
\noindent
Now suppose $h_j$ is the first gap among the marked gaps to be split at time $T$. Since all other marked gaps are still intact just before time $T$, after the split of $h_j$ there is an $r$-block consisting of the two pieces of $h_j$ together with $r-2$ adjacent marked gaps. Its total length is at most
\[
    (r-1)\alpha M_n.
\]
Thus
\[
    m_T\leq (r-1)\alpha M_n.
\]
If also $R_T\leq\rho$, then
\[
    M_T\leq \rho m_T
    \leq
    \rho(r-1)\alpha M_n
    =
    (r-1)(\rho-1+\eta)M_n.
\]
This is \eqref{eq:release-factor}.
\end{proof}

\smallskip
\noindent
The interpretation is that the marked block created at a slow split (a split that does not decrease $M_n$ substantially) is frozen until $M_n$ has fallen by the multiplicative factor $\beta$.

\section{Counting Protected Blocks}

We now convert Lemma \ref{lem:protected-block} into a global obstruction.

\bigskip
\noindent
Assume that
\[
    R_n\leq \rho
\]
for all sufficiently large $n$. Choose $\eta>0$ and define
\[
    \beta=(r-1)(\rho-1+\eta).
\]
We shall eventually choose $\rho$ and $\eta$ so that
\[
    \beta<\frac{r}{r+1}.          \tag{4.1}\label{eq:beta-small}
\]
Call a step $n\to n+1$ \emph{slow} if
\[
    M_{n+1}\geq (1-\eta)M_n,
\]
and \emph{fast} otherwise.

\bigskip
\noindent
Fix a large time $N$ so that $R_n\leq\rho$ for all $n\geq N$, and let $N^+$ be the first time after $N$ such that
\[
    M_{N^+}\leq \beta M_N.
\]
Such a time exists by \eqref{eq:M-upper-under-ratio}, since $M_n\to0$ under the eventual assumption $R_n\leq\rho$.

\bigskip
\noindent
The following proposition is the main counting step:

\begin{proposition}\label{prop:epoch}
There is a constant $C=C(r,\eta,\rho)$ such that, for all sufficiently large $N$,
\[
    N^+\leq \left(1+\frac1r\right)N+C .
\]
\end{proposition}

\begin{proof}
Before time $N^+$, we have
\[
    M_n>\beta M_N .
\]
Consider a slow step $k\to k+1$ with $N\leq k<N^+-1$. It creates a marked block as in Lemma \ref{lem:protected-block}, with $M_k\leq M_N$. By the lemma, none of the marked gaps can be split at any later time $T<N^+$, since otherwise
\[
    M_T\leq \beta M_k\leq \beta M_N,
\]
contradicting the definition of $N^+$. Thus every slow split before the terminal step $N^+-1\to N^+$ creates a protected block that remains untouched until time $N^+$.

\bigskip
\noindent
First we bound the number of fast steps before the terminal step. Suppose there are $b$ fast steps among
\[
    N\to N+1\to\dots\to N^+-2\to N^+-1.
\]
At every fast step $M_n$ is multiplied by a factor less than $1-\eta$, while at all other steps it does not increase. Therefore
\[
    M_{N^+-1}\leq (1-\eta)^b M_N.
\]
But by minimality of $N^+$,
\[
    M_{N^+-1}>\beta M_N.
\]
Consequently
\[
    (1-\eta)^b>\beta,
\]
so
\[
    b\leq B_0:=\left\lceil\frac{\log \beta}{\log(1-\eta)}\right\rceil .
\]
Here $B_0$ depends only on $r,\eta,\rho$.

\bigskip
\noindent
It remains to count the slow steps. We compare them with the $N$ gaps present at time $N$. Let $F$ be the set of initial gaps whose descendants are split during a fast step before the terminal step. Then $|F|\leq B_0$. Call an initial gap ``bad" if its cyclic distance, in the initial $N$-cycle, from some gap in $F$ is at most $r$. There are $O_r(B_0)$ bad initial gaps.

\bigskip
\noindent
The slow steps associated to bad initial gaps contribute only $O_{r,\eta,\rho}(1)$. Indeed, initially there are only $O_r(B_0)$ bad gaps. During the epoch, a fast split can increase by at most one the number of active descendants of bad gaps, and there are at most $B_0$ fast splits. A slow split of an active descendant removes it from further consideration, since its children lie in the protected block created by that slow split and cannot be split again before $N^+$. Thus the number of slow splits associated to bad initial gaps is $O_r(B_0)+B_0=O_{r,\eta,\rho}(1)$.

\bigskip
\noindent
For each remaining slow split, associate it to the unique initial gap whose descendant is split. We claim that these associated initial gaps are pairwise at cyclic distance at least $r$. Suppose not. Then two remaining slow splits are associated to good initial gaps $I$ and $J$ whose cyclic distance is at most $r-1$, allowing $I=J$. Let $A$ be the shorter cyclic arc of initial gaps from $I$ to $J$. Then $A$ contains at most $r$ initial gaps. Since $I$ and $J$ are good, no gap in $A$ has a descendant split during a fast step before the terminal step $N^+-1\to N^+$.

\bigskip
\noindent
Let $k\to k+1$ be the first slow split before the terminal step associated to an initial gap in $A$. Since no gap in $A$ has previously been split by either a fast or a slow step, all gaps of $A$ are still intact immediately before time $k$. At least one of the two chosen slow splits occurs after time $k$; let $J'\in A$ be the initial gap associated to such a later split. At time $k$, the gap $J'$ either is the split gap itself or lies among the $r-1$ gaps to the left or the $r-1$ gaps to the right of the split gap. Hence, after the split at time $k$, the descendant of $J'$ lies inside the protected block created at time $k$. This descendant therefore cannot be split before $N^+$, a contradiction.

\bigskip
\noindent
Thus, apart from $O_{r,\eta,\rho}(1)$ discarded exceptions, the initial gaps associated to slow split locations form a subset of an $N$-cycle with mutual cyclic distance at least $r$. Such a subset has size at most $N/r$. Therefore the number of slow steps before the terminal step is at most
\[
    \frac Nr+O_{r,\eta,\rho}(1).
\]
Adding the bounded number of fast steps and the terminal step $N^+-1\to N^+$ gives
\[
    N^+-N
    \leq
    \frac Nr+O_{r,\eta,\rho}(1),
\]
which proves the proposition.
\end{proof}

\section{Proof of the Main Theorem}

We now prove Theorem \ref{thm:main}.

\begin{proof}
Suppose for the sake of contradiction that
\[
    \limsup_{n\to\infty} R_n
    <
    1+\frac{r}{r^2-1}.
\]
Choose $\rho$ such that
\[
    \limsup_{n\to\infty}R_n<\rho<1+\frac{r}{r^2-1}.
\]
Then $R_n\leq\rho$ for all sufficiently large $n$.

\bigskip
\noindent
By the definition of $\rho$ we may choose $\eta>0$ so small that
\[
    \beta=(r-1)(\rho-1+\eta)<\frac{r}{r+1}.
\]
Starting from a sufficiently large $N_0$, define recursively
\[
    N_{j+1}=N_j^+,
\]
where $N_j^+$ is the first time after $N_j$ for which
\[
    M_{N_j^+}\leq \beta M_{N_j}.
\]
By Proposition \ref{prop:epoch}, there exists $C_0=C(r,\eta,\rho)$ such that
\[
    N_{j+1}\leq \left(1+\frac1r\right)N_j+C_0.
\]
It follows that, for some constant $C_1$,
\[
    N_j\leq C_1\left(1+\frac1r\right)^j .
\]
Therefore, using the average lower bound $M_n\geq r/n$,
\[
    M_{N_j}
    \geq
    \frac{r}{N_j}
    \geq
    C_2\left(\frac{r}{r+1}\right)^j
\]
for some $C_2>0$.

\bigskip
\noindent
On the other hand, by construction,
\[
    M_{N_j}\leq \beta^j M_{N_0}.
\]
Since
\[
    \beta<\frac{r}{r+1},
\]
these two estimates are incompatible for sufficiently large $j$. This contradiction proves
\[
    \limsup_{n\to\infty}R_n
    \geq
    1+\frac{r}{r^2-1}.
\]
\end{proof}

\section{Concluding Remarks}

The argument gives only a small asymptotic improvement over the de Bruijn--Erd\H{o}s lower bound:
\[
    1+\frac{r}{r^2-1}
    =
    1+\frac1r+\frac{1}{r(r^2-1)}.
\]
Thus it does not approach the logarithmic scale appearing in the upper construction of Cl\'ement and Steinerberger \cite{clement-steinerberger}. Nevertheless, the proof suggests that the splitting dynamics contains more information than the original first-moment obstruction.

\bigskip
\noindent
For $r=2$, the theorem gives
\[
    \limsup_{n\to\infty}\frac{M_n^{(2)}}{m_n^{(2)}}\geq \frac53.
\]
Numerical experiments suggest that the sharp value may be larger. In fact, we record the following sharp conjecture:

\smallskip
\begin{conjecture}\label{conj:r2}
For every sequence of distinct points on $\mathbb T$,
\[
    \limsup_{n\to\infty}
    \frac{M_n^{(2)}}{m_n^{(2)}}
    \geq 2.
\]
\end{conjecture}

\end{document}